\documentclass{amsart}
\usepackage{graphicx,amsmath,mathtools,amssymb,epsfig}
\usepackage{hyperref,epstopdf}
\newtheorem{theorem}{Theorem}[section]
\newtheorem{lemma}[theorem]{Lemma}
\newtheorem{definition}[theorem]{Definition}

\newtheorem{proposition}[theorem]{Proposition}
\newtheorem{remark}[theorem]{Remark}
\newtheorem{corollary}[theorem]{Corollary}
\newtheorem{conjecture}[theorem]{Conjecture}
\numberwithin{equation}{section}
\def \hfillx {\hspace*{-\textwidth} \hfill}

\begin{document}
	
\title{Homology Of Associative Shelves}
	
\author{Alissa S. Crans}
\thanks{Alissa S. Crans was supported by a grant from the Simons Foundation (\#360097, Alissa Crans).}
\address{Department of Mathematics, Loyola Marymount University, Los Angeles, California}
\email{acrans@lmu.edu}
		
\author{Sujoy Mukherjee}
\thanks{Sujoy Mukherjee was supported by the Presidential Merit Fellowship of the George Washington University}
\address{Department of Mathematics, George Washington University, Washington DC}
\email{sujoymukherjee@gwu.edu}
		
\author{J\'ozef H. Przytycki}
\thanks{J\'ozef H. Przytycki was partially supported by the Simons Foundation Collaboration Grant for Mathematicians-$316446$.}
\address{Department of Mathematics, George Washington University, Washington DC and University of Gda\'nsk}
\email{przytyck@gwu.edu}
		
	
\date{December 28, 2015 and, in revised form, March 28, 2016.}
	
	
\begin{abstract}
Homology theories for associative algebraic structures are well established and have been studied for a long time. 
More recently, homology theories for self-distributive algebraic structures motivated by knot theory, such as quandles and their relatives, have been developed and investigated. 
In this paper, we study associative self-distributive algebraic structures and their one-term and two-term (rack) homology groups.
\end{abstract}
	
\maketitle

\tableofcontents

In May of 2015, during a discussion of the origin of the word `quandle' on the {\it n-Category Caf\'{e}} (\url{https://golem.ph.utexas.edu/category/}), Sam C.\footnote{Unfortunately we do not know his full name. In \cite{SamC}, he writes: ``I am not a mathematician or physicist..."} remarked that unital shelves are associative and provided an elegant short proof. The work herein was inspired by this post. In Section \ref{intro}, we provide definitions and discuss some elementary propositions. The first two of these propositions were proven by Sam C. (\cite{SamC}). The latter propositions lay the pathway for the subsequent sections. Following this, we discuss a way to construct families of algebraic structures, in particular associative shelves and spindles.
	
Section \ref{finitelygen} addresses finitely generated associative shelves, their cardinalities and generating functions. Following this, finitely generated proto-unital shelves,  pre-unital shelves, and unital shelves are discussed. While some of the observations in this section are new, others are stated for completeness (\cite{JK}, \cite{Kep} and \cite{Zej}).
	
In Section \ref{homologyassoc}, the two well-known homology theories for self-distributive algebraic structures, namely one-term distributive homology and two-term (rack) homology are defined and the main results concerning the homology groups of a special class of associative shelves are proven. In particular, we compute one-term and two-term (rack) homology of unital shelves. The main theorem concerning two-term homology of shelves with right fixed elements extends to the two-term (rack) homology groups of Laver tables (\cite{Deh}, \cite{DL} and \cite{Leb}) and f-block spindles (\cite{CPP}, \cite{PS}).  

Finally, Section \ref{future} and the Appendix contain conjectures and computational data. 
	
\section{Introduction} \label{intro}

\begin{definition}\label{1.1}
	A {\bf quandle} $(X,*)$ is a magma \footnote{The term `magma' was introduced by Jean-Pierre Serre in 1965. The older term `groupoid,' introduced by \O{}ystein Ore in 1937, now has different meaning: it is a category in which  every morphism is invertible.} satisfying the following properties:
	\begin{enumerate}
		\item{\bf (idempotence)} $a*a = a$ for all $a\in X$,
		\item{\bf (inverse)} There exists $\bar{*}:X \times X \longrightarrow X$ such that $(a\bar{*}b)*b = (a*b)\bar{*}b$ for all $a,b \in X$, and
		\item{\bf (self-distributivity)} $(a*b)*c = (a*c)*(b*c)$ for all $a,b,c \in X$.
	\end{enumerate}
\end{definition} 

The primordial example of a quandle arises when starting with a group $(G,\cdot)$ and defining the quandle operation as $a*b = b^{-1}\cdot a \cdot b$ for all $a,b\in G$.   The three properties satisfied by the elements of a quandle are the algebraic analogues of the conditions satisfied by quandle colorings under the three Reidemeister moves as depicted in Figures ~\ref{fig 1} and \ref{fig 2}.

\begin{definition}
	By relaxing some of the axioms of a quandle, we obtain generalizations of quandles, including:
	\begin{enumerate}
		\item A {\bf rack} is a magma which satisfies the first and second quandle axioms. 
		\item A {\bf spindle} is a magma which satisfies the first and the third axioms of a quandle.
		\item A {\bf shelf} is a magma which satisfies only the third quandle axiom.
		\item A {\bf unital shelf} is a shelf $(X,*)$ equipped with an element $1$, such that $x*1=x=1*x$ for all $x\in X$.
	\end{enumerate}
\end{definition}
Table \ref{assocspindle} depicts an associative spindle that is not a quandle. We continue with some basic observations, due to Sam C.

\begin{figure}[tb]
	\includegraphics[scale=0.85]{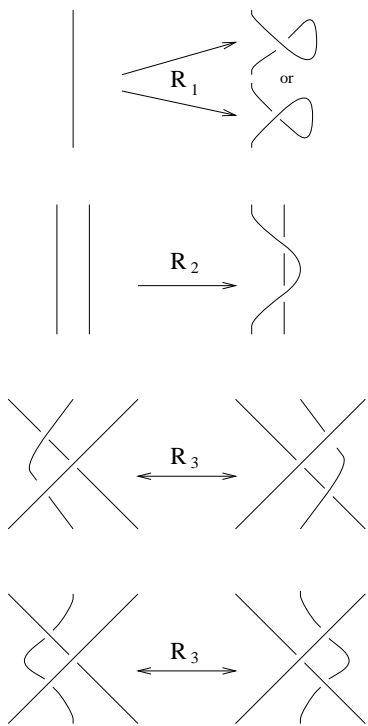}
	\caption{Unoriented Reidemeister moves but all coherent orientations could be added.}
	\label{fig 1}
\end{figure}

\begin{figure}[tb]
	\includegraphics[scale=0.50]{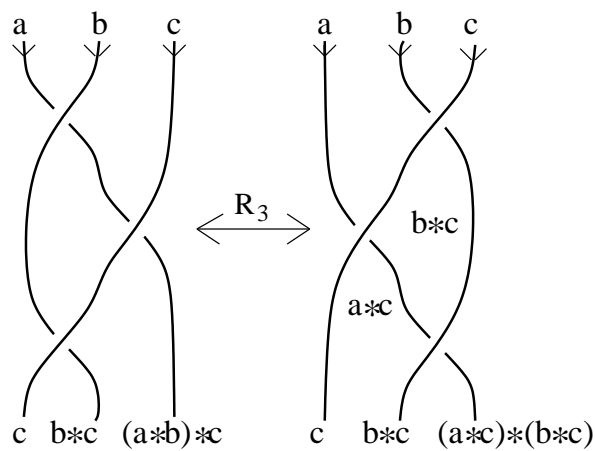}
	\caption{Quandle coloring for third Reidemeister move.}
	\label{fig 2}
\end{figure}

\begin{table}[ht] \label{assocspindle}
	\caption{An associative spindle}\label{1}
	\renewcommand\arraystretch{1}
	\noindent\[
	\begin{array}{c|c c c c}
	{*}&{a}&{b}&{c}&{d}\\
	\hline
	{a}&{a}&{b}&{c}&{d}\\
	{b}&{a}&{a}&{c}&{d}\\
	{c}&{a}&{a}&{c}&{d}\\
	{d}&{a}&{d}&{c}&{d}\\
	\end{array}
	\]
\end{table}
	
\begin{proposition}[Sam C.]\label{1.4}
	Let $(X,*)$ be a unital shelf. Then:
	\begin{enumerate}
		\item[(0)] For all $x\in X, x*x=x$
		\item[(1)] For all $a, b \in X,  a * b = b * (a * b)$
		\item[(2)] For all $a, b \in X, a * b = (a * b) * b$
			\end{enumerate}
\end{proposition}
	
\begin{proof}
	Let $a,b \in X$, where $X$ is a unital shelf. Then:
	\begin{enumerate}
		\item[(0)]{$a = (1 * 1) * a = (1 * a) * (1 * a) = a * a$}
		\item[(1)]{$a * b = (1 * a) * b = (1 * b) * (a * b) = b * (a * b)$} 
		\item[(2)]{$a * b = (a * 1) * b = (a * b) * (1 * b) = (a * b) * b$}
	\end{enumerate}
\end{proof}
	
\begin{remark}
	 \
	\begin{enumerate}
		\item[(1)] Monoids satisfying condition (1) of Proposition \ref{1.4} are known as `graphic monoids' and have been studied in \cite{Shu} and \cite{Kim}.\footnote{It can be easily shown that graphic monoids are self-distributive so that unital shelves and graphic monoids are the same algebraic structures. In \cite{Kim}, the number of elements in the free graphic monoid was computed. The name was coined by F. W. Lawvere (\cite{Law}).}
		\item[(2)] In Proposition \ref{1.4}, the proofs of parts (0) and (1) require only the left-sided unity, whereas, 
for the proof of part (2), both the left-sided and right-sided units are needed.
	\end{enumerate}
\end{remark}
	
\begin{definition}
	A shelf satisfying axioms (0), (1), and (2) of Proposition \ref{1.4} is called a {\bf pre-unital shelf}, and a shelf satisfying properties (1) and (2) of Proposition \ref{1.4} is called a {\bf proto-unital shelf}.
\end{definition}
	
\begin{corollary}[Sam C.]
	Proto-unital shelves are associative. In particular, this implies unital shelves are associative.
\end{corollary}
	
\begin{proof}
	Let $a,b,c \in X$, where $X$ is a unital shelf.  Then, $$(a * b) * c = (a * c) * (b * c) = (a * (b * c)) * (c * (b * c)) = (a * (b * c)) * (b * c) = a * (b * c)$$
\end{proof}

We conclude by illustrating a relationship between pre-unital and unital shelves.  	
\begin{proposition} \label{oneelt}
	If $(X, *)$ is a pre-unital shelf, then $(X \cup \lbrace 1 \rbrace,*)$ is unital. Conversely, if $(X,*)$ is a unital shelf, then $(X \setminus \lbrace 1 \rbrace ,*)$ is pre-unital. 
\end{proposition}

\begin{proof}
Extending a pre-unital shelf to a unital shelf requires verifying the self-distributivity property for all $a,b,c \in X$ where at least one of these elements is the unit. This follows easily. Conversely, for two elements $a,b \in  X$, neither of which are the unit, $a*b \neq 1$. Therefore, the result follows.	
\end{proof}

In the sections that follow, we study various properties of associative shelves but not of associative racks and quandles, even though these two algebraic structures are more important from knot-theoretic point of view. This is due to the fact that associate racks are trivial.
If $(X,*)$ is an associative rack with $a,b \in X$, then $(a*b)*b=(a*b)*(b*b)$. But this implies, $((a*b)*b)\bar{*}b)\bar{*}b = (((a*b)*(b*b))\bar{*}b)\bar{*}b$. As $(X,*)$ is associative, we may ignore the parentheses. Hence, it follows that $a=a*b$.

\begin{definition} \label{1.9}
	A {\bf quasi-group} $(X,*)$ is a magma satisfying:
	\begin{enumerate}
		\item{The equation $a * x = b$ has a unique solution in $X$ for all $a,b\in X$.}
		\item{The equation $x * a = b$ has a unique solution in $X$ for all $a,b\in X$.}
	\end{enumerate}
\end{definition}

Finite quasi-groups are also known as `Latin squares' and were introduced by Euler in the late 1700's. The two axioms of a Latin square $X$ of order $n$ essentially mean that the rows and columns of the multiplication table are permutations of the $n$ elements in $X$. Table \ref{2} illustrates a Latin square of size 4. Quandles satisfying the first axiom of Definition \ref{1.9} are called {\bf quasi-group quandles}.

\begin{table}[ht]
	\caption{A quasi-group}\label{2}
	\renewcommand\arraystretch{1}
	\noindent\[
	\begin{array}{c|c c c c}
	{*}&{a}&{b}&{c}&{d}\\
	\hline
	{a}&{a}&{c}&{d}&{b}\\
	{b}&{d}&{b}&{a}&{c}\\
	{c}&{b}&{d}&{c}&{a}\\
	{d}&{c}&{a}&{b}&{d}\\
	\end{array}
	\]
\end{table}
	
\begin{proposition}
	Let $(X,*)$ be an associative quasi-group. If $(X,*)$ satisfies condition $(2)$ of Proposition \ref{1.4} (in particular all elements of $X$ are idempotent), then $X$ has only one element.
\end{proposition}
	
\begin{proof}
	Let $a,b \in X$. Then we have $(a * b) * b = a * (b * b) = a * b$. But this implies $a * b = a$. Hence, as $b$ was arbitrary, there is only one element in $X$.
\end{proof} 

Note that this result only requires the second property of quasi-groups.
	
Using {\bf Microsoft Visual C++}, preliminary computations produce the data in Table \ref{3} regarding the number of associative shelves (AS), associative spindles (ASp), and unital shelves (US) of up to 4 elements. However, much more has been computed for semi-groups satisfying different combinations of axioms (\cite{JK}). 
	
\begin{table}[ht]
	\caption{Preliminary computations}\label{3}
	\renewcommand\arraystretch{1}
	\noindent\[
	\begin{array}{|c|c|c|c|}
	\hline 
	{n}&{\# AS}&{\# ASp}&{\# US}\\
	\hline
	{1}&{1}&{1}&{1}\\
	\hline 
	{2}&{4}&{3}&{1}\\
	\hline 
	{3}&{16}&{9}&{4}\\
	\hline 
	{4}&{93}&{38}&{6}\\
	\hline
	\end{array}
	\]
\end{table}

\section{Finitely Generated Associative Shelves} \label{finitelygen}

In this section, we consider finitely generated objects and prove that they are finite. Moreover, for these objects, we find normal forms, count number of elements and provide generating functions.

We will denote our alphabet by $A$, the monoid of words in alphabet $A$ with juxtaposition as the operation by $A^\star$, and the unit element of the monoid, namely the empty word, by $\emptyset$. If we delete the empty word we obtain the semi-group $A^\star\setminus \lbrace \emptyset \rbrace$.

\subsection{Free associative shelves}
	
\begin{definition}
		
	Let $A_n = \lbrace a_1,a_2,\cdots,a_n \rbrace$. Then the {\bf free associative shelf} on $A_n$, denoted by $FAS(n),$ is given by:  
	\begin{equation*}
		FAS(n) = \frac{A_n^\star\setminus \lbrace \emptyset \rbrace}{\sim}
	\end{equation*} where '$\sim$' is an equivalence relation defined by: Two words in $FAS(n)$ are equivalent under '$\sim$' if one can be transformed into the other by a finite number of right self-distributive operations with respect to juxtaposition.
\end{definition}
	
Note that it is enough to define '$\sim$' on the generators of $FAS(n)$. This can be proven by induction on the length of the words by considering three arbitrary words.

The structure of $FAS(n)$ was studied by a Czech group from Prague. The extensive survey of their work can be found in \cite{JK}. 

Intuition suggests that to obtain a word in its reduced form, it is enough to consider reductive self-distributive operations. However, this is not often the case.
 
We can reduce any word in ${A_n^\star\setminus \lbrace \emptyset \rbrace}$ uniquely to the normal form via the process described below. The crucial observation is that doing so requires the relation $ab^2a \sim aba$:
$$abba=a((bb)a)=a((ba)(ba))=((ab)(ab))a=((aa)b)a=(aa)(ba)=(ab)a=aba$$
The above equations illustrate how associativity and right self-distributivity go hand in hand.
	
For small values of $n$, computations produce:
\begin{equation*}
	FAS(1)=\lbrace a,a^2,a^3 \rbrace, \textrm{and}
\end{equation*}
\begin{equation*}
	FAS(2)=\lbrace a,b,ab,ba,a^2,b^2,b^3,ab^2,bab,b^2a,a^2b,aba,ba^2,a^3,bab^2,aba^2,a^2b^2,b^2a^2 \rbrace
\end{equation*}
Table \ref{FAS(2)} displays the multiplication table of $FAS(2)$ replacing the above words with integers from $0$ to $17$ in the order above. Similar calculations reveal the size of $FAS(3)$ as 93, which motivates the following proposition and theorem. Much more computational data appears in \cite{JK}.
{\footnotesize
	\begin{table}[ht]
		\caption{FAS(2)}\label{FAS(2)}
		\renewcommand{\arraystretch}{1}
		\noindent\[
		\begin{array}{c|ccc ccc ccc|ccc ccc ccc}						
{\star}&{0}&{1}&{2}&{3}&{4}&{5}&{6}&{7}&{8}&{9}&{10}&{11}&{12}&{13}&{14}&{15}&{16}&{17}\\
		\hline
		{0}&{4}&{2}&{10}&{11}&{13}&{7}&{7}&{16}&{10}&{11}&{10}&{11}&{15}&{13}&{16}&{15}&{16}&{15}\\
		{1}&{3}&{5}&{8}&{9}&{12}&{6}&{6}&{14}&{8}&{9}&{8}&{9}&{17}&{12}&{14}&{17}&{14}&{17}\\
		{2}&{11}&{7}&{10}&{11}&{15}&{7}&{7}&{16}&{10}&{11}&{10}&{11}&{15}&{15}&{16}&{15}&{16}&{15}\\
		{3}&{12}&{8}&{8}&{9}&{12}&{14}&{14}&{14}&{8}&{9}&{8}&{9}&{17}&{12}&{14}&{17}&{14}&{17}\\
		{4}&{13}&{10}&{10}&{11}&{13}&{16}&{16}&{16}&{10}&{11}&{10}&{11}&{15}&{13}&{16}&{15}&{16}&{15}\\
		{5}&{9}&{6}&{8}&{9}&{17}&{6}&{6}&{14}&{8}&{9}&{8}&{9}&{17}&{17}&{14}&{17}&{14}&{17}\\
		{6}&{9}&{6}&{8}&{9}&{17}&{6}&{6}&{14}&{8}&{9}&{8}&{9}&{17}&{17}&{14}&{17}&{14}&{17}\\
		{7}&{11}&{7}&{10}&{11}&{15}&{7}&{7}&{16}&{10}&{11}&{10}&{11}&{15}&{15}&{16}&{15}&{16}&{15}\\
		{8}&{9}&{14}&{8}&{9}&{17}&{14}&{14}&{14}&{8}&{9}&{8}&{9}&{17}&{17}&{14}&{17}&{14}&{17}\\
		\hline
		{9}&{17}&{8}&{8}&{9}&{17}&{14}&{14}&{14}&{8}&{9}&{8}&{9}&{17}&{17}&{14}&{17}&{14}&{17}\\
		{10}&{11}&{16}&{10}&{11}&{15}&{16}&{16}&{16}&{10}&{11}&{10}&{11}&{15}&{15}&{16}&{15}&{16}&{15}\\
		{11}&{15}&{10}&{10}&{11}&{15}&{16}&{16}&{16}&{10}&{11}&{10}&{11}&{15}&{15}&{16}&{15}&{16}&{15}\\
		{12}&{12}&{8}&{8}&{9}&{12}&{14}&{14}&{14}&{8}&{9}&{8}&{9}&{17}&{12}&{14}&{17}&{14}&{17}\\
		{13}&{13}&{10}&{10}&{11}&{13}&{16}&{16}&{16}&{10}&{11}&{10}&{11}&{15}&{13}&{16}&{15}&{16}&{15}\\
		{14}&{9}&{14}&{8}&{9}&{17}&{14}&{14}&{14}&{8}&{9}&{8}&{9}&{17}&{17}&{14}&{17}&{14}&{17}\\
		{15}&{15}&{10}&{10}&{11}&{15}&{16}&{16}&{16}&{10}&{11}&{10}&{11}&{15}&{15}&{16}&{15}&{16}&{15}\\
		{16}&{11}&{16}&{10}&{11}&{15}&{16}&{16}&{16}&{10}&{11}&{10}&{11}&{15}&{15}&{16}&{15}&{16}&{15}\\
		{17}&{17}&{8}&{8}&{9}&{17}&{14}&{14}&{14}&{8}&{9}&{8}&{9}&{17}&{17}&{14}&{17}&{14}&{17}\\
		\end{array}
		\]
\end{table} 
}

\begin{proposition} \label{normal}
	Words in normal form in $FAS(n)$ are characterized by the following family of words:
	\begin{enumerate}
		\item Non-empty words in which every letter appears at most once,
		\item Words of the form $aw$, where $w$ is a word in which every letter appears at most once and the letter $a$ appears exactly once, and
		\item Words obtained from words in (2) by doubling the last letter, in particular the word $a^3$ are possible.
	\end{enumerate} 
\end{proposition}

\begin{theorem} Let $| FAS(n) |$ be denoted by $c_n$.  Then:
	\begin{enumerate}
		\item  $c_n$ is finite with cardinality $3n+\sum_{i=2}^n(i+1)! {n \choose i}$
		\item$c_n$ satisfies the recursive relation $c_n= \frac{n^2}{n-1} c_{n-1} + n(n-1)$
		\item $c_n$ satisfies the $2$-term recursive relation $c_n= (n+2)c_{n-1} -(n-1)c_{n-2} + 3n$
		\item $c_n$ has exponential generating function $\frac{(3x-x^3)e^x}{(1-x)^2}$
	\end{enumerate}
\end{theorem}
		
\begin{proof} 
\
	\begin{enumerate}
	\item[(1)] This follows from Proposition \ref{normal} and \cite{JK}.
	\item[(2)] This follows from the formula given in \cite{BDG}, Appendix 2.\footnote{It is noted in that paper:``All 
        statements given here are proved in \cite{Gua}."} 
	\item[(3)] This follows from (2):
	$c_n  = \frac{n^2}{n-1} c_{n-1} + n(n-1)  =  (n+2)c_{n-1} + \frac{2-n}{n-1}c_{n-1} + n(n-1)  \stackrel{(2)}{=}  (n+2)c_{n-1} - (n-1)c_{n-2} + 3n.$	
	\item[(4)] This is straightforward from the theory of generating functions.
	\end{enumerate}
\end{proof}

\subsection{Free proto-unital shelves}\label{Subsection 3.3}

\begin{definition}
		Let $A_n = \lbrace a_1,a_2,\cdots,a_n \rbrace$. Then the {\bf free proto-unital shelf} on $A_n$, denoted by $FPUS(n)$, is given by: 
		\begin{equation*}
		FPUS(n) = \frac{A_n^\star\setminus \lbrace \emptyset \rbrace}{\sim}
		\end{equation*} where '$\sim$' is an equivalence relation defined by: Two words in $FPUS(n)$ are equivalent under '$\sim$' if one can be transformed into the other using a finite number of the axioms of proto-unital shelves with respect to juxtaposition.
\end{definition} 
As noted before, any proto-unital shelf is associative but not vice versa. 
%\begin{example} 
The free proto-unital shelf $FPUS(2)$ has six elements: $a,a^2, b, b^2, ab, ba$ and is depicted in Table \ref{FPUS(2)}.
\begin{table}[ht]
		\caption{$FPUS(2)$}\label{FPUS(2)}
$ \begin{array}{l | ccccccc}
	*    & a    & a^2   & ab    & b    & b^2    & ba    \\ \hline
	a    & a^2  & a^2   & ab    & ab   & ab     & ba     \\ 
	a^2  & a^2  & a^2   & ab    & ab   & ab     & ba      \\
	ab   & ba   & ba    & ab    & ab   & ab     & ba       \\
	b    & ba   & ba    & ab    & b^2  & b^2    & ba        \\ 
	b^2  & ba   & ba    & ab    & b^2  & b^2    & ba         \\
	ba   & ba   & ba    & ab    & ab   &  ab    & ba          \\
	\end{array}  $
	\end{table}
%\end{example} 

\begin{theorem}  Let $FPUS(A)$ be the free proto-unital shelf over the alphabet $A$. Then:
\begin{enumerate}
\item[(1)]  Every word in $FPUS(A)$ can be uniquely reduced to a word of type $a$, $a^2$ or $w$, where $w$ has at least two letters and no letter is repeated in $w$, 
\item[(2)] If $A=A_n$ is a finite alphabet of $n$ letters then, 
$$|FPUS_n|= n+\sum_{k=1}^n k!{n \choose k}= n-1 + \sum_{k=0}^n k!{n \choose k}$$
\item[(3)] The exponential generating function of $|FPUS(n)|$ is $(2x-x^2)\frac{e^x}{1-x}$
\end{enumerate}
\end{theorem}

\begin{proof}
	\
\begin{enumerate}
	\item[(1)] This follows from Lemma \ref{Lemma 3.10}.
	\item[(2)] This follows from counting elements in normal form.
		\item[(3)] This follows from a straightforward generating functions computation.
\end{enumerate}

\end{proof}

\begin{lemma}\label{Lemma 3.10}
 Let $A$ be a set (alphabet). Then the free proto-unital shelf $FPUS(A)$ is obtained from the semigroup of words over $A$ (that is $A^* \setminus \{\emptyset \}$), by adding the following relations:
 \begin{enumerate}
 	\item $ab^2=ab$ for any $a,b\in A$
 	\item $bab= ab$ for any $a,b\in A$
 	\item $awa=wa$ for any $a\in A$ and $w\in A^*  \setminus \{\emptyset\}$
 \end{enumerate}
	In particular, every element of the free proto-unital shelf over the alphabet $A$ has a unique representative of type $a$, $a^2$ or $w$, where $w$ has at least two letters and no letter is repeated.
\end{lemma}
\begin{proof}
	The main ingredient used here is the reduction of the word $a^2b$, because the number of letters actually increases before reducing:
	$$a^2b= b(a^2b)= ((ba)a)b= b(ab)=ab$$
With the identity $a^2b=ab$, the other reductions are immediate and a straightforward computation shows that words in normal (reduced) form with juxtaposition create a proto-unital shelf.

Consider the subset $A^* \setminus \{ \emptyset\},$ composed of letters $a^2$ and $w$ without repetition. The product here is defined  by the composite $w_1w_2$ where if $w_1=a=w_2$ for some $a\in A$, then $w_1w_2=a^2$ and otherwise if $w_1w_2$  uses more than one letter, we reduce it by deleting any letter which is not the last one of the type.   For example, $(a_1a_2a_4a_3)(a_1a_3)$ is reduced to $a_2a_4a_1a_3$. By construction, this subset is the largest possible so is thus $FPUS(A)$.
\end{proof}

\begin{definition}
	Let $(X,*)$ be a magma. If there exists $c,r \in X$ such that $x*r=c$ for all $x\in X$ then $r$ is a {\bf right $c$-fixed element}. If $r=c$ then $r(=c)$ is a {\bf right zero} or {\bf right projector}.
\end{definition}

\begin{corollary}
	$FPUS(n)$, has $n!$ different right zeros (right projectors), namely:
	\begin{enumerate}
		\item $a^2$ for $n=1$, and
		\item $a_{i_1}a_{i_2}\cdots a_{i_n}$ where $a_{i_1},a_{i_2},..., a_{i_n}$ is a permutation of alphabet letters from $A_n$.
	\end{enumerate}
	
\end{corollary}
\begin{proof} This follows from the structure of multiplication of words in normal form as described in the proof of Lemma \ref{Lemma 3.10}.
\end{proof}

\subsection{Free pre-unital and unital shelves}

Free pre-unital shelves and unital shelves are discussed together because, by the result of Proposition \ref{oneelt}, they differ only by one element (namely the empty word in the presentation of the shelf by generators and relators).
 
\begin{definition}
Let $A$ be an alphabet.  Then the {\bf free pre-unital shelf} on A, denoted by $F\widetilde{P}US(A)$, is given by:
	\begin{equation*}
	F\widetilde{P}US(A) = \frac{A^\star\setminus \lbrace \emptyset \rbrace}{\sim}
	\end{equation*} where '$\sim$' is an equivalence relation defined by: Two words in $F\widetilde{P}US(A)$ are equivalent under '$\sim$' if one can be transformed to the other using a finite number of the axioms of pre-unital shelves with respect to juxtaposition. If $A=A_n$ is a finite set of $n$ elements, we write $F\widetilde{P}US(n)$ for $F\widetilde{P}US(A)$.
\end{definition}
The free pre-unital shelf $F\widetilde{P}US(2)$ is illustrated in Table \ref{FPtildeUS(2)}.
\begin{table}[ht]
\caption{$F\widetilde{P}US(2)$}\label{FPtildeUS(2)}
$ \begin{array}{c | cccc}
*_1 & a  & b   & ab  & ba \\ \hline
a   & a  & ab  & ab & ba  \\ 
b   & ba & b   & ab & ba  \\
ab  & ba & ab  & ab & ba \\
ba  & ba & ab  & ab & ba
\end{array}  $
\end{table}

For any alphabet, $A$, we construct a pre-unital shelf $F\widetilde{P}US(A)$ and show that it is the free pre-unital shelf over alphabet $A$, using the observation that pre-unital shelves are associative.

\begin{definition} Let $A$ be any set and $F\widetilde{P}US(A)$ is the set of of non-empty words in $A^*$ which have no repeating letters. We define the {\bf binary operation $*$ on $F\widetilde{P}US(A)$} by: If $w_1,w_2 \in  F\widetilde{P}US(A)$ then $w_1*w_2$ is obtained from $w_1w_2$ (juxtaposition) by deleting from $w_1$ all letters which are also in $w_2$.  
\end{definition}
\begin{proposition}\label{2.14} 
	$F\widetilde{P}US(A)$ is a pre-unital shelf isomorphic to the free pre-unital shelf over alphabet $A$.
\end{proposition}  

\begin{proof}
	
First we show that $F\widetilde{P}US(A)$  is a pre-unital shelf. Self-distributivity and associativity are verified directly from definition. Namely, $(w_1*w_2)*w_3$ is obtained from $w_1w_2w_3$ by deleting from $w_1$ letters which are already in $w_2$ and deleting from $w_1*w_2$ letters which are already in $w_3$. When using the rules for $*$ operation we see that $w_1*(w_2*w_3)$ and  $(w_1*w_3)*(w_2*w_3)$ give the same result.

Furthermore, from the rule $w_1*w_1=w_1$, idempotence follows immediately as does axiom (1) in Proposition \ref{1.4}: $w_2*(w_1*w_2)= w_1*w_2$. Axiom (2) in Proposition \ref{1.4} follows from associativity and idempotence: $(w_1*w_2)*w_2= w_1*(w_2*w_2)= w_1*w_2$.
\end{proof}

 From Proposition \ref{2.14} we can calculate the number of elements in $F\widetilde{P}US(n)$ and its exponential generating function.

\begin{corollary}
	Let $|F\widetilde{P}US_n|$ be denoted by $b_n$.  Then:
	\begin{enumerate} 
		\item[(1)] $b_n = \sum_{k=1}^n n(n-1)...(n-k+1)=  \sum_{k=1}^n k!{n \choose k}$
		\item[(2)] $b_n$ satisfies the recursion relation
		$b_n = nb_{n-1} +n$ with  $b_0=0,b_1=1, b_2=4, b_3=15, b_4=64, b_5= 325, b_6=1956. $
		\item[(3)] $b_n$ has exponential generating function
		$\sum_{k=0}^{\infty} b_k\frac{x^k}{k!} = (\frac{e^x}{1-x} - e^x) = e^x(\frac{x}{1-x})$
	\end{enumerate}
\end{corollary}
\begin{proof} 
	\
	\begin{enumerate}
	\item[(1)] The irreducible (normal form) words in $F\widetilde{P}US(n)$ (equivalently words in $F\widetilde{P}US(n)$ are all arrangements of $k$ different letters of the alphabet of $n$ letters (empty word excluded). We calculate the number as: $n+ n(n-1)+...+ n(n-1)...(n-k+1)+...+ n!= \sum_{k=1}^n k!{n \choose k}$, as desired. 
	\item[(2)] This follows directly from the formula in (1).
	\item[(3)] This follows from a straightforward generating functions computation.
	\end{enumerate}
	
\end{proof}

The free unital shelf $FUS(n)$ is obtained from $F\widetilde{P}US(n)$ using Proposition \ref{oneelt} by adding empty word. Therefore, we have the following:

\begin{proposition}
	\
	\begin{enumerate}
		\item The number of elements in $FUS(n)$ is one more than the number of elements in $F\widetilde{P}US(n)$. That is, $| FUS(n) |  = \sum_{k=0}^n k! {n\choose k}$
			
		\item The exponential generating function for $| FUS(n)|$ is  $\frac{e^x}{1-x}$
	\end{enumerate}
\end{proposition}

\begin{remark}
The material of this subsection was essentially proven by Kimura in 1958 \cite{Kim} where he considered monoids satisfying the axiom $a*b = b*a*b$ (where $a$ can be the empty word). It was proven that free monoids with this axiom consist of words without repeating letters, hence in the case of $n$ generators, he obtained $\sum_{k=0}^n k!{n \choose k}$ elements.
\end{remark}

\section{Homology Of Associative Shelves} \label{homologyassoc}
		
In this section, we first recall the definition of one-term distributive homology introduced in \cite{Prz} and two-term (rack) homology introduced in \cite{FRS1,FRS2,FRS3}. Following \cite{Lod}, we first define the chain modules (denoted by $C_n$) and the face maps $d_{i,n}:C_n \longrightarrow C_{n-1}$, satisfying $d_{i,n-1}\circ d_{j,n} = d_{j-1,n-1}\circ d_{i,n}$ for $0 \leq i \leq n$ and $i<j$, so that $(C_n,d_{i,n})$ forms a pre-simplicial module. After defining the boundary map $\partial_n:C_n \longrightarrow C_{n-1}$ by $\partial_n = \sum_{i=0}^n (-1)^i d_{i,n}$, we obtain the chain complex $\mathcal{C}=(C_n,\partial_n)$.  
		
We describe two pre-simplicial modules leading up to one-term and two-term (rack) homology, respectively. Let $(X,*)$ be a shelf and $\mathcal{C}_n$ the free abelian group generated by the $(n+1)$-tuples $(x_0,x_1,\ldots,x_n)$ of the elements in $X$, that is $\mathcal{C}_n= \mathbb{Z}X^{n+1}$. 

We begin with the pre-simplicial module leading to one-term homology. For $0 \leq i \leq n$, let $ d^{(*)}_{i,n} : \mathcal{C}_n \longrightarrow \mathcal{C}_{n-1}$ be given by:

\begin{equation*}
	d^{(*)}_{i,n}((x_0,x_1,\cdots,x_n))=
	\begin{cases*}
	(x_1,x_2,\ldots,x_n) & if $i = 0$,\\
	(x_0*x_i,x_1*x_i,\ldots,x_{i-1}*x_{i},x_{i+1},\ldots,x_n) & $0 < i < n$, \\
	(x_0*x_n,x_1*x_n,\ldots,x_{n-1}*x_n) & if $i = n$.
	
	\end{cases*}
\end{equation*}
We then have  $d^{(*)}_{i,n-1} \circ d^{(*)}_{j,n} = d^{(*)}_{j-1,n-1} \circ d^{(*)}_{i,n}$, and therefore $(C_n,d_{i,n}^{(*)} )$ is a pre-simplicial module. Defining the boundary map as usual by $\partial_n^{(*)}= \sum_{i=0}^{n} (-1)^i  d^{(*)}_{i,n}$, then $\mathfrak{C}^{(*)} = (\mathcal{C}_n, \partial_n^{(*)})$ is a chain complex whose $n^{th}$ homology group of $X$ is given by:
\begin{equation*}
H_n^{(*)}(X) = \frac{ker(\partial_n^{(*)})}{im(\partial_{n+1}^{(*)})}.
\end{equation*} 

To define two-term (rack) homology, we consider the trivial operation given by $a*_0b=a $. Notice that $d^{(*_0)}_{i,n}((x_0,x_1,\ldots,x_n))=(x_0,x_1,\ldots,x_{i-1},x_{i+1},x_{i+2},\ldots,x_n).$  We define face maps $d_{i,n}^R:C_n \longrightarrow C_{n-1}$ by $d_{i,n}^R=d_{i,n}^{(*_0)} - d_{i,n}^{(*)}$. We then have $d_{i,n-1}^R \circ d_{j,n}^R = d_{j-1,n-1}^R \circ d_{i,n}^R$, hence $(C_n,d_{i,n}^R)$ is a pre-simplicial module.  Defining the boundary map as usual by $\partial_n^{R}= \sum_{i=0}^{n} (-1)^i  d^{R}_{i,n},$ then $\mathfrak{C}^R = (\mathcal{C}_n, \partial_n^R)$ is a chain complex whose $n^{th}$ two-term (rack) homology group of $X$ is given by: 
\begin{equation*}
H_n^R(X) = \frac{ker(\partial_n^R)}{im(\partial_{n+1}^R)}.
\end{equation*}

We now introduce the main results of the paper on one-term and two-term (rack) homology of unital shelves. The following theorem in the context of one-term homology appears in \cite{Prz} and in \cite{PS}.
	
\begin{theorem}[Prz]\label{3.1}
	If one of the following conditions holds:
	\begin{enumerate}
		\item {$x\longmapsto x*y$ is a bijection on $X$ for some $y$, or}
		\item{there is $y \in X$ such that $y*x=y$ for all $x \in X$,} \\
	\end{enumerate} then $\widetilde{H}_n^{(*)}(X) = 0 $ for all $n$.\footnote{$\widetilde{H}^{(*)}$ denotes the augmented one-term homology satisfying: 
	\begin{equation*}
	H^{(*)}_n(X)=
	\begin{cases*}
	\mathbb{Z}\oplus \widetilde{H}^{(*)}(X)  & if $n = 0$,\\
	\widetilde{H}^{(*)}(X)  & if $n \geq 1$.
	\end{cases*}
	\end{equation*}}
\end{theorem}
	
There are many associative shelves satisfying condition $(2)$ of Theorem \ref{3.1}. All of them have trivial one-term homology in all dimensions. Furthermore, applying Theorem \ref{3.1} gives:
	
\begin{theorem}
	Unital shelves have trivial one-term homology in all positive dimensions.   That is $\widetilde{H}^{(*)}(X)=0$, for all $n \geq 0$.
\end{theorem}
		
\begin{proof} The existence of the unit element in unital shelves implies condition $(1)$ of Theorem \ref{3.1} is satisfied. 
\end{proof}		
	
A similar, but stronger result also holds for two-term homology. Before stating this result, namely Theorem \ref{main theorem}, we remark that in a shelf $(X,*)$ if $r$ is a right $c$-fixed element, then $c = c*r = (c*r)*r = (c*r)*(r*r) = c*c$. Therefore, $c$ is idempotent. 

\begin{theorem} \label{main theorem}
	Two-term (rack) homology of shelves with a right fixed element is $\mathbb{Z}$ in all dimensions.
\end{theorem} 
	
The proof of the above theorem follows from the following three lemmas.
	
\begin{lemma} \label{lemma1}
		
	Let $(X,*)$ be a shelf with a right $c$-fixed element $r$. Let $\mathcal{C}=(C_n,\partial_n)$ be the chain complex corresponding to the two-term (rack) homology. Then, there exists a sequence of maps $f_n:\mathbb{Z}X^n \longrightarrow \mathbb{Z}X^{n+1}$ which is a chain homotopy between the identity on $C_n$ and the  constant map on $C_n, \alpha^c_n:C_n\longrightarrow C_n$, given by $\alpha_n^c(x_0,x_1,...,x_n)=(c,c,...,c)$.
\end{lemma}
	
\begin{proof}
	Let $(X,*)$ be a shelf with an idempotent element $c$ and a right $c$-fixed element $r$.  Let $\mathfrak{C} = (\mathcal{C}_n, \partial_n)$ be the chain complex corresponding to two-term (rack) homology of $(X,*)$. Consider a sequence of maps (chain homotopy) $f_n:\mathbb{Z}X^n \longrightarrow \mathbb{Z}X^{n+1}$, for all positive integers $n$, given by:
	\begin{equation*}
		f(x_0,x_1,...,x_n)=(-1)^{n+1}(x_0,x_1,...,x_n,r).
	\end{equation*}
	  We will show that this sequence of maps is a chain homotopy between the identity chain map and the constant chain map $(x_0,x_1,...,x_n) \longmapsto (c,c,...,c)$.
		
Consider the complex:

\begin{figure}[h]
	\includegraphics[scale=0.20]{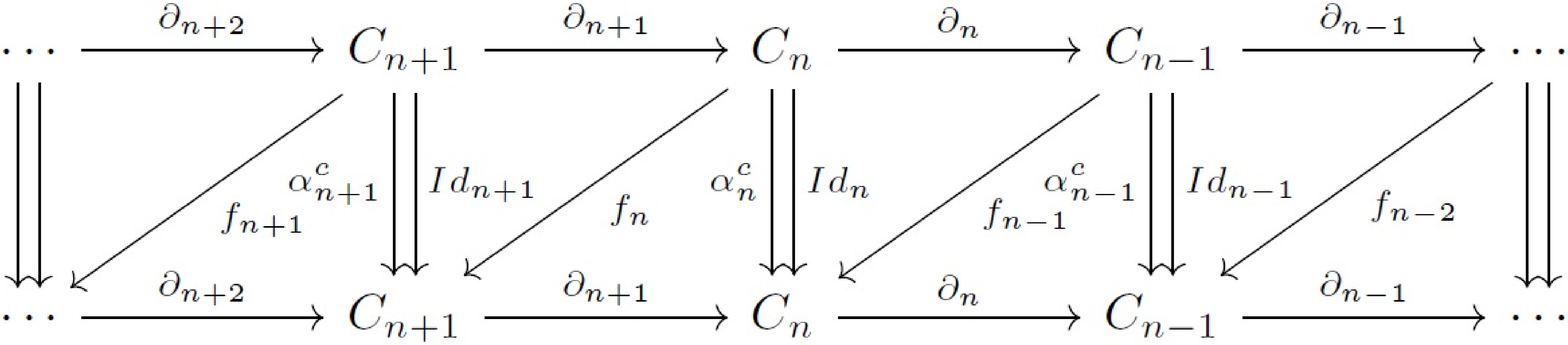}
\end{figure}

Then:
$$\begin{multlined}
		(\partial_{n+1} \circ f_n)(x_0,x_1,...,x_n)=(-1)^{n+1}\partial_{n+1}(x_0,x_1,...,x_n,r)\\
		=(-1)^{n+1} \{ \sum_{i=1}^n (-1)^i(x_0,x_1,...,x_{i-1},x_{i+1},x_{i+2},...,x_n,r)\}+(x_0,x_1,...,x_n)\\
		-(-1)^{n+1}\{ \sum_{i=1}^n (-1)^i(x_0*x_i,x_1*x_i,...,x_{i-1}*x_{i},x_{i+1},x_{i+2},...,x_n,r)\}\\
		-(c,c,...,c),
		\end{multlined}$$ and 
	
		$$\begin{multlined}
		(f_{n-1} \circ \partial_n)(x_0,x_1,...,x_n) =f_{n-1}(\sum_{i=1}^n (-1)^i(x_0,x_1,...,x_{i-1},x_{i+1},x_{i+2},...,x_n)\\
		-\sum_{i=1}^n (-1)^i(x_0*x_i,x_1*x_i,...,x_{i-1}*x_{i},x_{i+1},x_{i+2},...,x_n) )\\
		=(-1)^n \{ \sum_{i=1}^n (-1)^i(x_0,x_1,...,x_{i-1},x_{i+1},x_{i+2},...,x_n,r)\}\\
		-(-1)^n\{\sum_{i=1}^n (-1)^i(x_0*x_i,x_1*x_i,...,x_{i-1}*x_{i},x_{i+1},x_{i+2},...,x_n,r)\}
		\end{multlined}$$ Taking the sum of the two equations above we get $$(\partial_{n+1} \circ f_n)(x_0,x_1,...,x_n) + (f_{n-1} \circ \partial_n)(x_0,x_1,...,x_n) = (x_0,x_1,...,x_n) - (c,c,...,c),$$ that is, $\partial \circ f = f \circ \partial = Id - \alpha^c$.
		
\end{proof}

\begin{lemma}
	Let $\mathcal{C}^c$ be a sub-chain complex of $\mathcal{C}$, where $C^c_n$ is spanned by $(c,c,...,c)$. Then, $H_n^R(\{c\}) = \mathbb{Z}$.
\end{lemma}

\begin{proof}
		
	We have $C^c_n = \mathbb{Z}$ and the boundary maps are zero maps. Therefore, $H_n^R(\{c\})=C^c_n = \mathbb{Z}$..
		
\end{proof}

\begin{lemma}
	$\mathcal{C}$ is chain homotopy equivalent to $\mathcal{C}^c$. In particular, they have the same homology groups.
\end{lemma}

\begin{proof}
	Let $\gamma^c_n: C_n \longrightarrow C_n^c$, be given by $\gamma^c_n(x_0,x_1,...,x_n) = (c,c,...,c)$ and let $\beta^c_n: C_n^c \longrightarrow C_n$ be an embedding given by $\beta^c_n(c,c,...,c) = (c,c,...,c)$. Then, $\gamma^c_n \circ \beta^c_n = Id_{C^c_n}$ and by Lemma \ref{lemma1}, $\beta^c_n \circ \gamma^c_n $ is homotopy equivalent to $Id_{C_n}$. Therefore, these chain complexes are chain homotopy equivalent.
\end{proof}
	
We continue with the observation:	
\begin{proposition}
	Finite proto-unital shelves have at least one right zero element.
\end{proposition}	
	
\begin{proof}
	Let $(X,*)$ be a finite proto-unital shelf with at least two elements and no right zeros. Then there exist two distinct elements $a_1,a_2 \in X$ such that $a_1$ is not a right zero and $a_2*a_1 \neq a_1$. Otherwise, the only possibility is $x*x \neq x$ for all $x \in X$ and $x*y=y$ for all $x \neq y \in X$. But this is a contradiction since $x,y\in X$ with $x \neq y$ implies $ y*(x*y) = y*y \neq y=x*y$. Therefore, $a_1,a_2\in X$, $a_1$ is not a right zero, $a_2 \neq a_1$ and $a_2*a_1 \neq a_1$.  
		
	Now, if $a_2*a_1$ is a right zero, we are done. Therefore, suppose it is not. Thus, there exists $a_3 \in  X$ such that $a_3*(a_2*a_1) \neq a_2*a_1$. Since $X$ is associative, parentheses can be ignored. This implies $a_3 \neq a_2$ since otherwise we have $a_3*a_2*a_1 = a_2*a_2*a_1 = a_2*a_1*a_2*a_1 = a_1*a_2*a_1 = a_2*a_1$. Due to same argument as above, $a_3 \neq a_2*a_1$. Also, $a_3 \neq a_1$, since otherwise we have $a_3*a_2*a_1 = a_1*a_2*a_1 = a_2*a_1$. Similarly, $a_3 \neq a_1*a_2$. A similar inductive argument shows that $a_3$ cannot be generated by $a_1,a_2$. An analogous argument  is used for the element $a_3*a_2*a_1$.
		
	Next, assume the element $a_n*a_{n-1}*\cdots*a_1$ is not a right zero for some finite $n$. This implies there exists $a_{n+1} \in X$ such that $a_{n+1}*(a_n*a_{n-1}*\cdots*a_1) \neq a_n*a_{n-1}*\cdots*a_1$. As above, due to associativity of $X$, parentheses can be ignored. Let $a_{n+1}=a_{b_1}*a_{b_2}*\cdots*a_{b_k}$, where $b_1,b_2,...,b_k \in \{1,2,...,n\}$ so that we have
$$a_{n+1}*a_n*a_{n-1}*\cdots*a_1=a_{b_1}*a_{b_2}*\cdots*a_{b_k}*a_n*a_{n-1}*\cdots*a_1$$ If $a_{b_1}=a_j$ for some $1 \leq j \leq n$, then:
$$a_{b_1}*a_{b_2}*\cdots*a_{b_k}*a_n*\cdots*a_{j+1}*a_j*\cdots*a_1=a_{b_2}*\cdots*a_{b_k}*a_n*\cdots*a_{j+1}*a_j*\cdots*a_1$$ The above equation follows from the second axiom of proto-unital shelves, $b*a*b = a*b$ for all $a,b \in X$. By the same argument, we can reduce $a_{b_2}$ followed by $a_{b_3}$ until $a_{b_{k-1}}$. Finally when $a_{b_k} = a_{n+1}$, there are two possibilities. Either $a_{b_k} \neq a_n$, or they are equal. In the first case, the above argument can be used. For the second case, the reduction requires more work. We have: 

\begin{eqnarray*} 
	a_n*a_n*a_{n-1}*\cdots*a_1 &=& a_n*a_{n-1}*a_n*a_{n-1}*a_{n-2}*\cdots*a_1 \\ &=& a_{n-1}*a_n*a_{n-1}*\cdots*a_1 \\ &=& a_n*a_{n-1}*\cdots*a_1.
\end{eqnarray*} 
This contradicts the fact that $a_{n+1}$ can be generated by the previous terms. Therefore, at every stage, it is possible to have a new element of $X$ which cannot be generated by the previous elements. But, this is a contradiction as we assumed $(X,*)$ was finite. Therefore, $(X,*)$ has at least one right zero element.
\end{proof}
Note that the converse of the above proposition is not true. Table \ref{converseprop} gives an example of an associative shelf that is not a proto-unital shelf, as we have $(3*1)*1 = 2*1 = 0\neq 2 = 3*1 $.

\begin{table}[ht]
	\caption{An associative shelf which is not a proto-unital shelf but has right zeros }\label{converseprop}
	\renewcommand\arraystretch{1}
	\noindent\[
	\begin{array}{c |c c c c}
	
	{*}&{0}&{1}&{2}&{3}\\
	\hline
	{0}&{0}&{0}&{2}&{3}\\
	{1}&{0}&{0}&{2}&{3}\\
	{2}&{0}&{0}&{2}&{3}\\
	{3}&{0}&{2}&{2}&{3}\\
	\end{array}
	\]
\end{table}
This brings us to the stronger analog of Theorem \ref{main theorem}:
\begin{theorem}
	Two-term (rack) homology of proto-unital shelves is $\mathbb{Z}$ in all dimensions.
\end{theorem}

For the free algebraic structures discussed in the previous section, we have the corollary:

\begin{corollary}
	\
	\begin{enumerate}
		\item $FUS(n)$ has trivial one-term homology in all positive dimensions for all $n$.
		\item $FPUS(n)$ has two-term (rack) homology $\mathbb{Z}$ in all dimensions for any $n$. Hence, in particular, the same holds for $F\widetilde{P}US(n)$ and $FUS(n)$ as well.  
	\end{enumerate}
\end{corollary}

We continue with two applications of Theorem \ref{main theorem}.  Recall that a spindle is a magma satisfying the idempotence and self-distributivity properties of a quandle.  

\begin{definition}[\cite{CPP}]
	Choose a family of sets $\{X_i\}_{i \in I}$, not necessarily finite, and functions $f_i : X_i  \longrightarrow X_i$. Define the spindle product on $ X = \bigsqcup _{i \in I} X_i$ for $x \in X_i$ and $y  \in X_j$ by:
	\begin{equation*}
	x*y=
	\begin{cases*}
	y, &  if i = j \\
	f_j(y),  &  otherwise
	\end{cases*}
	\end{equation*}
Subsets $X_i \subset X$ are called {\bf blocks} of the spindle $X$ and the maps $f_i$ are called {\bf block functions}. We write $f : X \longrightarrow X$ for the function induced by all the block functions.
\end{definition}

One-term homology was initially thought to be torsion-free and hence less interesting than two-term (rack) homology. In \cite{PS}, there are several examples in which families of shelves are constructed. In \cite{CPP}, the one-term homology of one such infinite family was studied. Theorem 5.4 in \cite{CPP} states that any finite abelian group can be realized as the torsion subgroup of the second homology group of $f$-block spindles having a one element block. But, every $f$-block spindle with a one element block has a right zero element. 
	
Therefore, there exist infinite families of spindles which can realize any finite commutative group as a torsion subgroup in one-term homology. But, they have $\mathbb{Z}$ as their two-term homology group in all dimensions.

Theorems \ref{3.1} and \ref{main theorem} can be further applied to determine the one-term and two-term (rack) homology groups of Laver tables. Richard Laver discovered these algebraic structures while studying set theory, and the combinatorial structure of Laver tables has been extensively studied (\cite{Deh}, \cite{DL} and \cite{Leb}).
	
In the following discussion, we use the convention of Laver(\cite{Lav}, \cite{Deh}, \cite{Leb}, \cite{DL}), that is the magmas are left self-distributive. In particular, to get right self-distributivity we transpose the Laver tables.

\begin{definition}
	A {\bf Laver table}, denoted by $A_n$, is the unique shelf $(\{1,2,\cdots,2^n\},*)$ for any positive integer $n$ satisfying $a*1 = 1+a (\textrm{\rm mod} \;  2^n)$.
\end{definition}
	
The uniqueness in this definition was proven by Richard Laver in 1992. Tables 8, 9, and 10 provide examples and satisfy left self-distributivity.

\begin{table}[h] 
\begin{minipage}{0.4\textwidth}
\centering
\begin{tabular}{c|c c c c} \label{10}
	{$*$}&{1}&{2}&{3}&{4}\\
	\hline
	{1}&{2}&{4}&{2}&{4} \\
	{2}&{3}&{4}&{3}&{4}\\
	{3}&{4}&{4}&{4}&{4} \\
	{4}&{1}&{2}&{3}&{4} \\
\end{tabular}
\caption{The Laver table $A_2$} \label{table7}
\end{minipage}
\hfillx
\begin{minipage}{0.6\textwidth}
\centering
\begin{tabular}{c|c c c c c c c c} 
	$*$&1&2&3&4&5&6&7&8\\
	\hline
	1&2&4&6&8&2&4&6&8\\
	2&3&4&7&8&3&4&7&8\\
	3&4&8&4&8&4&8&4&8\\
	4&5&6&7&8&5&6&7&8\\
	5&6&8&6&8&6&8&6&8\\
	6&7&8&7&8&7&8&7&8\\
	7&8&8&8&8&8&8&8&8\\
	8&1&2&3&4&5&6&7&8\\
\end{tabular}
\caption{The Laver table $A_3$} \label{table8}
\end{minipage}
\end{table}
\begin{table}
\centering
\begin{tabular}{c | c c c c  c c c c | c c c c c c c c}
	$*$&1&2&3&4&5&6&7&8&9&10&11&12&13&14&15&16\\
	\hline
	1&2&12&14&16&2&12&14&16&2&12&14&16&2&12&14&16\\
	2&3&12&15&16&3&12&15&16&3&12&15&16&3&12&15&16\\
	3&4&8&12&16&4&8&12&16&4&8&12&16&4&8&12&16\\
	4&5&6&7&8&13&14&15&16&5&6&7&8&13&14&15&16\\
	5&6&8&14&16&6&8&14&16&6&8&14&16&6&8&14&16\\
	6&7&8&15&16&7&8&15&16&7&8&15&16&7&8&15&16\\
	7&8&16&8&16&8&16&8&16&8&16&8&16&8&16&8&16\\
	8&9&10&11&12&13&14&15&16&9&10&11&12&13&14&15&16\\
	\hline
	9&10&12&14&16&10&12&14&16&10&12&14&16&10&12&14&16\\
	10&11&12&15&16&11&12&15&16&11&12&15&16&11&12&15&16\\
	11&12&16&12&16&12&16&12&16&12&16&12&16&12&16&12&16\\
	12&13&14&15&16&13&14&15&16&13&14&15&16&13&14&15&16\\
	13&14&16&14&16&14&16&14&16&14&16&14&16&14&16&14&16\\
	14&15&16&15&16&15&16&15&16&15&16&15&16&15&16&15&16\\
	15&16&16&16&16&16&16&16&16&16&16&16&16&16&16&16&16\\
	16&1&2&3&4&5&6&7&8&9&10&11&12&13&14&15&16\\
\end{tabular}
\caption{The Laver table $A_4$} \label{table9}
\end{table}
	
The element $2^{n-1}$ in the Laver table $A_n$ (after taking transpose) is a right fixed element. Further, the element $2^n$ is an element with respect to which $A_n$ (after taking transpose) satisfies condition $(1)$ of Theorem \ref{3.1} (\cite{Deh,Leb}). Therefore, the following result holds:
	
\begin{corollary}
	Laver tables, after converting to right self-distributive magmas (by taking transpose), have trivial reduced one-term homology groups in all dimensions and their two-term (rack) homology groups are  $\mathbb{Z}$ in all dimensions. 
\end{corollary}
		
\section{Future directions} \label{future}
	
 The results contained here regarding one-term and two-term (rack) homology for special families of associative shelves naturally lead to a desire to better understand the role of associativity in self-distributive algebraic structures. These preliminary computations lead to the conjectures that follow. Additionally some computational data appears in the Appendix.
	
\begin{conjecture}
	Associative shelves have no torsion in one-term homology.
\end{conjecture}
		
\begin{conjecture}
	Associative shelves have no torsion in two-term (rack) homology.
\end{conjecture}

\section{Acknowledgements}
	
The authors would like to thank Maciej Niebrzydowski for initial computations and Marithania Silvero for useful comments and suggestions.
	
\section{Appendix: Computer Calculations}

\subsection{One-term homology}

The work done in previous sections demonstrated that one-term homology for shelves with a left zero or a right unit is trivial for every positive dimension. The first three one-term homology groups for some associative shelves of size four not satisfying these conditions appear below:

\begin{table}[h]
	\begin{minipage}{0.5\textwidth}
		\centering
		\begin{tabular}{c | c c c c}
			$*$&0&1&2&3 \\
			\hline
			0&0&0&0&3\\
			1&0&0&0&3\\
			2&0&0&2&3\\
			3&0&0&3&3\\	
		\end{tabular}
		\caption{$T_5$}
	\end{minipage}
	\hfillx
	\begin{minipage}{0.5\textwidth}
		\centering
		\begin{tabular}{c|c c c c}
			$*$&0&1&2&3\\
			\hline
			0&0&0&2&2\\
			1&1&1&3&3\\
			2&0&0&2&2\\
			3&1&1&3&3\\
		\end{tabular}
		\caption{$T_6$}
	\end{minipage}
\end{table}

\begin{table}[h]
	\begin{minipage}{0.5\textwidth}
		\centering
		\begin{tabular}{c | c c c c}
			$*$&0&1&2&3 \\
			\hline
			0&0&0&2&3 \\
			1&0&1&2&3\\
			2&0&0&2&3\\
			3&0&0&2&3\\	
		\end{tabular}
		\caption{$T_7$}
	\end{minipage}
	\hfillx
	\begin{minipage}{0.5\textwidth}
		\centering
		\begin{tabular}{c | c c c c}
			$*$&0&1&2&3 \\
			\hline
			0&0&1&2&3 \\
			1&0&1&2&3\\
			2&0&1&2&3\\
			3&0&1&2&3\\	
		\end{tabular}
		\caption{$T_8$}
	\end{minipage}
\end{table}

The one-term homology groups are:
\begin{enumerate}
	\item{$H_0^{(*)}(T_5)=\mathbb{Z}, H_1^{(*)}(T_5)=0, H_2^{(*)}(T_5)=0$.}
		\vspace*{2pt}
	\item{$H_0^{(*)}(T_6)=\mathbb{Z}^2, H_1^{(*)}(T_6)=\mathbb{Z}^4, H_2^{(*)}(T_6)=\mathbb{Z}^{16}$.}
		\vspace*{2pt}
	\item{$H_0^{(*)}(T_7)=\mathbb{Z}^3, H_1^{(*)}(T_7)=\mathbb{Z}^8, H_2^{(*)}(T_7)=\mathbb{Z}^{32}$.}
		\vspace*{2pt}
	\item{$H_0^{(*)}(T_8)=\mathbb{Z}^4, H_1^{(*)}(T_8)=\mathbb{Z}^{12}, H_2^{(*)}(T_8)=\mathbb{Z}^{48}$.}
\end{enumerate}

\subsection{Two-term (rack) homology}

The main result, Theorem \ref{main theorem},  completely describes the two-term (rack) homology groups of shelves with right-fixed elements. However, there exist associative shelves without right fixed elements. The first three one-term homology groups for some associative shelves of size four without right fixed elements, together with the first three homology groups of a unital shelf, appear below:

\begin{table}[h]
	\begin{minipage}{0.5\textwidth}
		\centering
		\begin{tabular}{c | c c c c}
			$*$&0&1&2&3 \\
			\hline
			0&0&0&0&0 \\
			1&0&1&1&1\\
			2&0&1&2&2\\
			3&0&1&2&3\\	
		\end{tabular}
		\caption{$T_1$}
	\end{minipage}
	\hfillx
	\begin{minipage}{0.5\textwidth}
		\centering
		\begin{tabular}{c|c c c c}
			$*$&0&1&2&3\\
			\hline
			0&0&0&0&0\\
			1&0&0&0&0\\
			2&0&0&0&0\\
			3&3&3&3&3\\
		\end{tabular}
		\caption{$T_2$}
	\end{minipage}
\end{table}

\begin{table}[h]
	\begin{minipage}{0.5\textwidth}
		\centering
		\begin{tabular}{c | c c c c}
			$*$&0&1&2&3 \\
			\hline
			0&0&0&0&0 \\
			1&0&0&0&0\\
			2&2&2&2&2\\
			3&3&3&3&3\\	
		\end{tabular}
		\caption{$T_3$}
	\end{minipage}
	\hfillx
	\begin{minipage}{0.5\textwidth}
		\centering
		\begin{tabular}{c | c c c c}
			$*$&0&1&2&3 \\
			\hline
			0&0&0&0&0 \\
			1&1&1&1&1\\
			2&2&2&2&2\\
			3&3&3&3&3\\	
		\end{tabular}
		\caption{$T_4$}
	\end{minipage}
\end{table}

The two-term (rack) homology groups are:
\begin{enumerate}
	\item{$H_0^R(T_1)=\mathbb{Z}, H_1^R(T_1)=\mathbb{Z}, H_2^R(T_1)=\mathbb{Z}$.}
	\vspace*{2pt}
	\item{$H_0^R(T_2)=\mathbb{Z}^2, H_1^R(T_2)=\mathbb{Z}^4, H_2^R(T_2)=\mathbb{Z}^8$.}
	\vspace*{2pt}
	\item{$H_0^R(T_3)=\mathbb{Z}^3, H_1^R(T_3)=\mathbb{Z}^9, H_2^R(T_3)=\mathbb{Z}^{27}$.}
	\vspace*{2pt}
	\item{$H_0^R(T_4)=\mathbb{Z}^4, H_1^R(T_4)=\mathbb{Z}^{16}, H_2^R(T_4)=\mathbb{Z}^{64}$.}
\end{enumerate}

It is not difficult to notice the pattern suggested from the data above, and thus the results for all other associative shelves of size four are not included.  This pattern holds for racks (\cite{EG, LN}). More computations on torsion in two-term homology for quandles were done in \cite{NP, PY} for example.

\end{document}